\documentclass[11pt]{amsart}
\usepackage{amsmath}
\usepackage{amssymb}
\usepackage{amsthm}
\usepackage{amsfonts}
\usepackage{bm}
\usepackage{bbm}
\usepackage{float}  %
\usepackage{xcolor}
\usepackage{dsfont}
\usepackage{tikz-cd}
\usepackage{mathrsfs}
\usepackage{mathtools}
\usepackage{graphicx}
\usepackage{xargs}

\newtheorem{thm}{Theorem}[section]
\newtheorem*{thm*}{Theorem}
\newtheorem{Thm}{Theorem}
\newtheorem{lm}[thm]{Lemma}
\newtheorem*{lm*}{Lemma}
\newtheorem{prop}[thm]{Proposition}
\newtheorem*{prop*}{Proposition}
\newtheorem{Prop}[Thm]{Proposition}
\newtheorem{cor}[thm]{Corollary}
\newtheorem*{cor*}{Corollary}
\newtheorem{Cor}[Thm]{Corollary}

\theoremstyle{definition}
\newtheorem{defi}[thm]{Definition}

\theoremstyle{remark}
\newtheorem{rmk}[thm]{Remark}
\newtheorem*{rmk*}{Remark}
\newtheorem{ex}[thm]{Example}

\newtheorem*{question}{Question}

\DeclareMathOperator{\rank}{rk}

\DeclareMathOperator{\lk}{lk}
\DeclareMathOperator*{\connsum}{\raisebox{-0ex}{\scalebox{1.4}{$\#$}}}
\DeclareMathOperator{\FF}{\mathbb{F}}
\DeclareMathOperator{\RR}{\mathbb{R}}
\DeclareMathOperator{\ZZ}{\mathbb{Z}}
\DeclareMathOperator{\Mo}{\widehat{\mathcal{M}}}
\DeclareMathOperator{\lfr}{\mathrm{rk}_{HFL}}

\graphicspath{ {./images/} }

\title{Links of second smallest knot Floer homology}

\author[Juhyun Kim]{Juhyun Kim}
\address{Department of Mathematics, Caltech, Pasadena, CA, 91125}
\email{jhyunk@caltech.edu}


\begin{document}

\maketitle

\begin{abstract}
We prove that the rank of knot Floer homology detects the Hopf links, and generalize this result further to classify the links of the second smallest knot Floer homology. We also prove a knot Floer homology analog of \cite[Theorem 1]{lipshitzKhovanovHomologyAlso2019} and give a partial answer to when the equality holds in the rank inequality between the knot Floer homology of a link and its sublinks.
\end{abstract}

\section{Introduction}
Since the introduction, knot Floer homology \cite{ozsvathHolomorphicDisksKnot2004, rasmussenFloerHomologyKnot2003} and its refinement, link Floer homology \cite{ozsvathHolomorphicDisksLink2008}, have been a powerful tool in knot and link detection problem. For example, knot Floer homology detects the unknot \cite{ozsvathHolomorphicDisksGenus2004}, trefoil knots \cite{ghigginiKnotFloerHomology2008, heddenGeographyBotanyKnot2018}, figure-eight knots \cite{ghigginiKnotFloerHomology2008}, and unlinks \cite{niHomologicalActionsSutured2014}. It is interesting to observe that, among the detection results listed above, all except the figure-eight knot can be detected by the rank of the knot Floer homology without further knowledge on the internal structure of their knot Floer homology. Similar pattern appears in other homological invariants like Khovanov homology; cf. \cite{batsonLinksplittingSpectralSequence2015}, \cite{baldwinKhovanovHomologyDetects2018} for instance.

In this paper, we investigate such rank detection theorems further with the focus on links of multiple components. 

\begin{Thm}\label{thm:main}
For a fixed $l$, the second smallest knot Floer homology rank of $l$-component links is $2^l$, and only the disjoint union of the Hopf link and the $(l-2)$-component unlink realizes this rank.
\end{Thm}
As a special case of Theorem \ref{thm:main}, we have
\begin{Cor}\label{cor:rankdetectionHopf}
The rank of knot Floer homology detects the Hopf link. More precisely, for a 2-component link $L$, $L$ is the Hopf link if and only if $\rank\widehat{HFK}(L)=4$.
\end{Cor}
\begin{rmk}
As the Hopf link is the only 2-component fibered link of genus 0 (in $S^3$), it is a direct consequence of the fiberedness detection theorem \cite{niKnotFloerHomology2007} that knot Floer homology detects the Hopf link, cf. \cite{binnsKnotFloerHomology2020}. The novelty here is that the rank of the knot Floer homology alone is strong enough to detect the Hopf link.
\end{rmk}

One key idea in the proof of Theorem \ref{thm:main} is the fact that, like many knot and link invariants, the computation of knot Floer homology of split link can be reduced to that of the factor sublinks:\\
For a split link $L\cong L_1\sqcup L_2$,
\begin{align*}
    \widehat{HFK}(L_1\sqcup L_2)\cong \widehat{HFK}(L_1)\otimes \widehat{HFK}(L_2)\otimes \FF^2,
\end{align*}
where $\FF=\FF_2$ is the prime field of characteristic 2.

In view of this ``divide and conquer" principle, it is important to have a criterion for the splitness of the link in terms of the knot Floer homology. Recently such a splitness detection result for Khovanov homology is proven by Lipshitz and Sarkar \cite[Theorem 1]{lipshitzKhovanovHomologyAlso2019} in terms of the natural module structure over the truncated polynomial ring $A=\FF[X]/(X^2)$ on the Khovanov homology. The module structure is defined as follows \cite{khovanovPatternsKnotCohomology2003, heddenKhovanovModuleDetection2013}:\\
Given a basepoint $p$ on $L$, the merge map with a small unknot in a neighborhood of $p$ defines an action of $p$ on the Khovanov chain complex of $L$, which induces the action of $p$ on the Khovanov homology. Taking a point $q$ at a different link component and considering the induced action of $p$ on the reduced Khovanov homology $\widetilde{Kh}(L, q)$, we obtain a $\FF[X]$-action on $\widetilde{Kh}(L, q)$ where $X$ acts as the basepoint action of $p$. It is easy to check that $X^2=0$, hence the reduced Khovanov homology $\widetilde{Kh}(L,q)$ is endowed with an $A$-module structure. 

There is a similar structure in the context of knot Floer homology which assigns for a pair of basepoints of $L$ an action of $A$, often called the homological action \cite{niHomologicalActionsSutured2014}. For a pair of basepoints of $L$, we connect the two basepoints with an arc. Viewed as a relative 1-cycle in the link complement, the homological action of this arc also induces a $A$-module structure on the knot Floer homology of $L$.

We claim a knot Floer homology version of the splitness detection result using this homological action on the knot Floer homology:
\begin{Thm}\label{thm:splitdetection}
Let $L$ be a link in $S^3$. Then the followings are equivalent:
\begin{enumerate}
\item $L$ is split.
\item $\widehat{HFK}(L)$ is free as an $A$-module, where the $A$-module structure is induced by a path $\zeta$ connecting two different link components.
\end{enumerate}
\end{Thm}

\begin{rmk}
In the course of this paper, Wang \cite{wangLinkFloerHomology2020} independently proves a similar result using surface decomposition theory. Our proof of Theorem \ref{thm:splitdetection} follows more closely on \cite{lipshitzKhovanovHomologyAlso2019} by checking the link Floer counterpart of deep results in \cite{alishahiBorderedFloerHomology2019}.
\end{rmk}

Thanks to the disjoint union formula for the knot Floer homology mentioned above, the classification of links whose knot Floer rank is within a certain bound can be reduced to the same question over nonsplit links. For example, we can consider the relative version of the rank detection problem in view of the following rank inequality:\\
For a link $L$ and a link component $L_1\subseteq L$,
\begin{align*}
    \rank \widehat{HFK}(L)\geq 2\cdot\rank \widehat{HFK}(L\setminus L_1).
\end{align*}
In the relative version of the rank detection problem, we seek a link component $L_1$ which satisfies the equality in the above inequality in the sense that a link component is considered minimal with respect to $L$ if the equality holds. We may hope that a link (under certain conditions) is rank-minimal if there is a sequence of link components such that each stage of successive removal of the components satisfies the equality. The unlink detection theorem of Ni \cite{niHomologicalActionsSutured2014}, viewed from this perspective, states that we can find a sequence of link components of full length if and only if the link is an unlink.

Using the formula for the link Floer homology of split links given above, the relative rank detection problem is completely solved when $L_1$ is unlinked from the rest of $L$:\\
When $L$ is a disjoint union of $L_1$ and $L\setminus L_1$, the equality holds if and only if $\rank \widehat{HFK}(L_1)=1$, which is by the unknot detection theorem equivalent to $L_1$ being an unknot.\\
An unlinked, unknotted component of $L$ is often referred to as a \textit{trivial} component.

Hence the question of when the equality holds above becomes interesting only when we restrict the links of consideration to links. In this context, we can give a partial answer:
\begin{Prop}\label{prop:relative}
Let $L$ be a nonsplit link of $l>1$ components, $L_1$ a component of $L$, and suppose that equality holds in the above rank inequality. Then,
\begin{itemize}
\item The component $L_1$ is algebraically unlinked from the rest of the link, ie, $\lk(L_1, L_i)=0$ for $i>1$.
\item The sublink $L\setminus L_1$ is nonsplit.
\end{itemize}
\end{Prop}

This paper is organized as follows. In Section 2, we review various constructions in link Floer homology used throughout this paper. In Section 3, we review the homological action on the link Floer homology and prove a link Floer homology version of the splitness detection theorem, Theorem \ref{thm:splitdetection}. In Section 4, we discuss some basic facts in convex geometry and its consequences on link Floer polytopes. With the results in Section 4, we prove Theorem \ref{thm:main} in Section 5. In Section 6, we study the conditions for equality in the rank inequality and prove Proposition \ref{prop:relative}.

\medskip \noindent \textit{Acknowledgements.} The author would like to thank Yi Ni for the suggestion of the problem and the advice throughout the project.

\section{Preliminary}
We assume that the reader is familiar with the basics of link Floer homology and sutured Floer homology in general; see \cite{rasmussenFloerHomologyKnot2003, ozsvathHolomorphicDisksKnot2004, ozsvathHolomorphicDisksLink2008, juhaszHolomorphicDiscsSutured2006} for the definitions and notations used throughout this section. Unless otherwise stated, the coefficient ring for all variants of Floer homology groups is $\FF=\mathbb{Z}/2\mathbb{Z}$.

\subsection{Knotification}
Let $L\subseteq Y$ be a null-homologous oriented link in a closed oriented three-manifold $Y$ and choose two points $p_1, q_1\in L$ which are not on the same component of $L$. Then one can form a new (null-homologous) link $\kappa(L, \left\{p_1,q_1\right\})\subseteq Y\# (S^1\times S^2)$ of one less components, constructed as follows:\\
First attach a 1-handle with the two feet on $p_1$ and $q_1$ over $S^3$. Then take the band sum along a band whose core is the core of the attached 1-handle.

\begin{center}
    \begin{figure}[ht]
    \includegraphics[width=0.8\columnwidth]{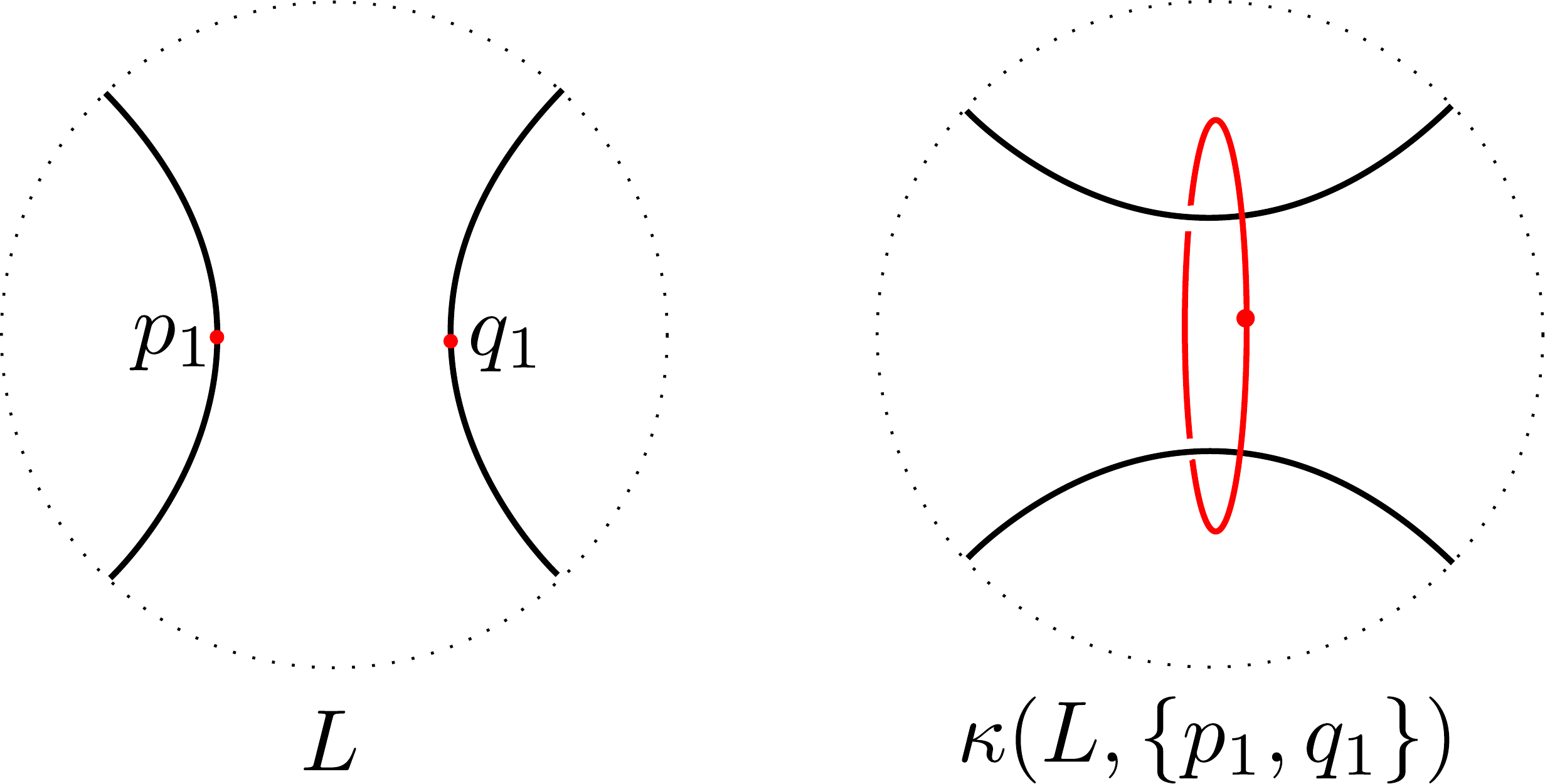}
    \caption{Local picture of knotification using dotted circle notation. A 1-handle with feet on $p_1$ and $q_1$ is added and the band along the core of the 1-handle is summed on the link $L$ to form a link $\kappa(L, \left\{p_1,q_1\right\})$ inside $\kappa(Y, \left\{p_1,q_1\right\})\cong Y\connsum (S^1\times S^2)$.}
    \label{fig:knotification}
    \end{figure}
\end{center}

This process can be repeated to yield a link $\kappa(L, \left\{p_i, q_i \right\})\subseteq Y\connsum^k{(S^1\times S^2)}$. When we repeat this process to the end, the resulting knot $\kappa(L)$ (or the process itself) is called \textit{knotification} \cite[Chapter 2.1]{ozsvathHolomorphicDisksKnot2004} and independent of the choice of the pairs of points and the band sums, justifying the omission of points $p_i, q_i$ in the notation. For the same reason, a one-step process as above only depends on the link components $p_i$'s and $q_i$'s are lying on.

We abuse the notation and define the following:
\begin{defi}
 The one-step process as above from $L\subseteq Y$ to $\kappa(L,\left\{p_1,q_1\right\})\subseteq Y\# (S^1\times S^2)$ is called \textit{knotification} (despite the fact that it may not result in a knot).
\end{defi}
\begin{rmk}
A knotification process only depends on the components that the feet of the 1-handle are lying on.
\end{rmk}
Throughout this paper, we restrict our attention to links in $S^3$ and their knotifications. There is a straightforward generalization of the link Floer homology (of links in $S^3$) to null-homologous links in $\connsum^n(S^1\times S^2)$, which we briefly review here. A curious reader may check \cite[Section 5]{zemkeLinkCobordismsAbsolute2019} for the details on the gradings of link Floer homology in general 3-manifolds.

By the construction, the link Floer homology groups as in \cite{ozsvathHolomorphicDisksLink2008} are endowed with a natural $spin^c$ grading. In case of links in $S^3$, there is a convenient identification of the $spin^c$ grading $Spin^c(L)$ with the affine lattice 
\begin{align*}
    \mathbb{H}(L):=\left\{\sum_i a_i[\mu_i]\middle| 2a_i+\lk(L\setminus L_i, L_i)\textrm{\;is\;even}\right\},    
\end{align*}
defined up to a global shift (to guarantee the central symmetry of the grading) by the following formula:
\begin{align*}
    Spin^c(L) &\to \mathbb{H}(L)\\ 
    \mathfrak{t} &\mapsto \sum_i \frac{\left<c_1(\mathfrak{t}), [F_i] \right>-\lk(L\setminus L_i, L_i)}{2}[\mu_i],
\end{align*}
where $F_i$ is a generalized Seifert surface such that $\partial F_i=L_i$. 
For general null-homologous links, the map $f:Spin^c(L)\to \mathbb{H}(L)$ defined as above is only surjective. But we can still pass the $spin^c$ grading through this map to assign a $\mathbb{H}(L)$-grading on the link Floer homology:\\
For $h\in \mathbb{H}(L)$,
\begin{align*}
    \widehat{HFL}(\connsum^n(S^1\times S^2), L, h) = \bigoplus_{f(\mathfrak{t})= h}\widehat{HFL}(\connsum^n(S^1\times S^2), L, \mathfrak{t}).
\end{align*}
This is the grading (Alexander multigrading) which we use throughout this paper.

The importance of the notion of knotification can be understood by the following theorem.
\begin{thm}[Generalization of {\cite[Theorem 1.1]{ozsvathHolomorphicDisksLink2008}}]\label{thm:hflofknotification}
There is a natural identification
\begin{align*}
    \widehat{HFL}(\kappa(L, \left\{p,q\right\}), h)\cong \bigoplus_{o(h')=h}\widehat{HFL}(L, h'),
\end{align*}
where $o:\mathbb{H}(L)\to \mathbb{H}(\kappa(L,\left\{p,q\right\}))$ is the homomorphism defined on the level of basis by the canonical identification $[\mu_i]\mapsto [\mu_i]$.
\begin{proof}
This is a straightforward adaptation of \cite[Proof of Theorem 1.1]{ozsvathHolomorphicDisksLink2008} where only a single 1-handle is attached on the Heegaard diagram for the link.
\end{proof}
\end{thm}
\begin{rmk}
\cite[Theorem 1.1]{ozsvathHolomorphicDisksLink2008} follows as a special case of Theorem \ref{thm:hflofknotification} because for each $s\in \ZZ$,
\begin{align*}
    \widehat{HFK}(L,s)\cong \widehat{HFK}(\kappa(L),s)\cong\widehat{HFL}(\kappa(L),s)
\end{align*}
by the definition of $\widehat{HFK}$.
\end{rmk}
In this regard, the rank of $\widehat{HFL}(L)$ is equal to the rank of $\widehat{HFK}(L)$, so there is no ambiguity in the terminology \textit{link Floer rank} of $L$.
\begin{defi}
For a link $L\subseteq S^3$, its \textit{link Floer rank} $\lfr(L)$ is the dimension of the $\FF$-vector space $\widehat{HFL}(L)$.
\end{defi}

\subsection{Properties of link Floer homology}
Here we collect several basic properties of the link Floer homology group $\widehat{HFL}(L,h)$ of a link $L\subseteq S^3$. For the ease of notation, we identify the element $h=\sum_i h_i[\mu_i]\in \mathbb{H}(L)$ with the tuple $(h_1,\cdots, h_l)$ of half-integers.
\begin{thm}[Central symmetry]
\begin{align*}
    \widehat{HFL}(L, h)\cong \widehat{HFL}(L, -h)
\end{align*}
holds for all $h\in \mathbb{H}(L)$.
\end{thm}

\begin{thm}[Decategorification]
For a $l$-component link $L\subseteq S^3$ with $l>1$,
\begin{align*}
    \chi(\widehat{HFL}(L)):=\sum_{d\in \ZZ, h\in \mathbb{H}}(-1)^d\rank\widehat{HFL}(L,h)\cdot t_1^{h_1}\cdots t_l^{h_l}=\prod_{i=1}^l(t_i^{\frac{1}{2}}-t_i^{-\frac{1}{2}})\cdot\Delta (L)
\end{align*}
holds where $\Delta(L)$ is the multi-variable Alexander polynomial of $L$.
\end{thm}
\begin{rmk}\label{rmk:even}
Especially, $\lfr(L)$ is even when $l>1$ and odd when $l=1$.
\end{rmk}

\begin{thm}[Connected sum formula, {\cite[Theorem 1.4]{ozsvathHolomorphicDisksLink2008}}]\label{thm:connsum}
For links $L_1,L_2\subseteq S^3$, we have the following isomorphisms:
\begin{align*}
    \widehat{HFL}(L_1\# L_2, h)&\cong \bigoplus_{h'\#h''=h}\left(\widehat{HFL}(L_1, h')\otimes \widehat{HFL}(L_2, h'')\right),\\
    \widehat{HFL}(L_1\sqcup L_2, h)&\cong \widehat{HFL}(L_1,h')\otimes \widehat{HFL}(L_2,h'')\otimes \FF^2,
\end{align*}
where in the first isomorphism $h'\#h''$ is the image of $(h', h'')$ under the natural map $\mathbb{H}(L_1)\otimes \mathbb{H}(L_2)\to \mathbb{H}(L)$ and in the second isomorphism $h'$(resp. $h''$) is the image of $h\in \mathbb{H}(L)$ under the natural projection map $\mathbb{H}(L)\to \mathbb{H}(L_1)$(resp. $\mathbb{H}(L)\to\mathbb{H}(L_2)$).
\end{thm}
\begin{thm}[Component removal spectral sequence, {\cite[Proposition 7.1]{ozsvathHolomorphicDisksLink2008}}]\label{thm:CFSS}
For a link $L$ and its component $L_1\subseteq L$, we can find a differential $D$ (that is, a homomorphism $D$ with $D^2=0$) filtered with respect to the first Alexander grading on $\widehat{HFL}(L)$ satisfying
\begin{align}
    H_*(\widehat{HFL}(L),D)\cong \widehat{HFL}(L\setminus L_1)\otimes \FF^2.
\end{align}
$D$ is homogeneous with respect to the other Alexander gradings in the sense that if $x\in \widehat{HFL}(L)$ has degree $h=(h_1,\cdots, h_l)\in \mathbb{H}(L)$, $D(x)$ has degree $(h_2-\frac{1}{2}\lk(L_1,L_2), \cdots, h_l-\frac{1}{2}\lk(L_1,L_l))\in \mathbb{H}(L\setminus L_1)$.
\end{thm}
\begin{rmk}
This can also be understood as the effect of removing a $z$ basepoint from the Heegaard diagram, see proof of Proposition \ref{prop:relative}.
\end{rmk}

\begin{cor}[Rank inequality]
For a $l$-component link $L$ and its component $L_1\subset L$,
\begin{align*}
    \lfr(L)\geq 2\cdot \lfr(L\setminus L_1).
\end{align*}
In particular, for a $l$-component link $L\subseteq S^3$,
\begin{align*}
    \lfr(L)\geq 2\cdot \lfr(L\setminus L_1)\geq \cdots \geq 2^{l-1}\cdot\lfr(L_l) \geq 2^{l-1}
\end{align*}
\cite[Theorem 1.2]{ozsvathHolomorphicDisksLink2008}.
\end{cor}
\begin{rmk}
\cite{niHomologicalActionsSutured2014} showed that only the $l$-component unlink achieves the equality in the second inequality.
\end{rmk}

\begin{rmk}\label{rmk:hopf}
In particular, if $L$ contains a knotted component, $\lfr(L)\geq 3\cdot 2^{l-1}>2^l$. As the disjoint union of the Hopf link and the $(l-2)$-component unlink has the link Floer rank $2^l$, such an $L$ does not have the second smallest link Floer rank.
\end{rmk}

\section{Homological action and detection of split link}
In this section, we recall the notion of homological action on the link Floer homology and link Floer homology with twisted coefficients. General references to this section are \cite[Chapter 8]{ozsvathHolomorphicDisksThreeManifold2004}, \cite{niNonseparatingSpheresTwisted2013} with the caveat that they only consider $\widehat{HF}$ and $\widehat{HFK}$ respectively, and in \cite{niNonseparatingSpheresTwisted2013} certain special case of the following construction.

\subsection{Homological action}
Suppose that $(\Sigma, \bm{\alpha},\bm{\beta},\mathbf{w},\mathbf{z})$ is a Heegaard diagram representing the link $L\subseteq S^3$. Then, by \cite{juhaszHolomorphicDiscsSutured2006}, the sutured Floer homology of the link exterior $SFH(S^3(L))$ is isomorphic to $\widehat{HFL}(L)$. Thus the homological action \cite{niHomologicalActionsSutured2014} on $SFH(S^3(L))$ defines the corresponding action on $\widehat{HFL}(L)$.

More precisely, let $\zeta\subseteq \Sigma$ be a relative 1-cycle on $\Sigma$. Then $\zeta$ defines a map $e_\zeta:\pi_2(\mathbf{x},\mathbf{y})\to \ZZ$ by the following formula:
\begin{align*}
    e_\zeta(\phi)=\zeta\cdot \partial_\alpha \phi,   
\end{align*}
where $\partial_\alpha\phi=(\partial \phi)\cap \mathbb{T}_\alpha$ is the components of $\partial \phi$ which lies on $\bm{\alpha}$, interpreted as a multi-arc in $\Sigma$. Using this map, $\zeta$ acts on the generator $\mathbf{x}\in \mathbb{T}_\alpha\cap \mathbb{T}_\beta$ of $\widehat{CFL}(L)$ as
\begin{align}\label{eq:0}
    \zeta\cdot \mathbf{x} = \sum_{\mathbf{y}\in \mathbb{T}_\alpha\cap \mathbb{T}_\beta}\sum_{\phi\in \pi_2(\mathbf{x},\mathbf{y})}\# \Mo(\phi)e_\zeta(\phi)\mathbf{y}.
\end{align}
The action of $\zeta$ commutes with the differential of $\widehat{CFL}(L)$, hence there is an induced action on $\widehat{HFL}(L)$. \cite{niHomologicalActionsSutured2014} proved that the action of $\zeta$ on $\widehat{HFL}(L)$ only depends on the homology class $[\zeta]\in H_1(S^3, L)$ and squares to zero.

The following theorem connects the homological action on $\widehat{HFL}(Y,L)$ to the homological action on its knotification. Before stating the theorem, we introduce the following notation.

For a relative 1-cycle $\zeta\in H_1(S^3, L)$ and basepoints $p, q\in L$ on two different components of $L$, the closure $\tilde{\zeta}\in H_1(\kappa(S^3, \left\{p,q\right\}), \kappa(L,\left\{p,q\right\}))$ of $\zeta$ is defined as follows:
\begin{itemize}
    \item For each component of $\zeta$ connecting the two components where $p$ and $q$ are lying on, the corresponding component of $\tilde{\zeta}$ is the union of the component of $\zeta$ (extended to a neighborhood of $p$ and $q$ if necessary) and the core of the 1-handle.
    \item For all the other components of $\zeta$, the corresponding components of $\tilde{\zeta}$ are the natural inclusion.
\end{itemize}
\begin{center}
    \begin{figure}[ht]
    \includegraphics[width=0.9\columnwidth]{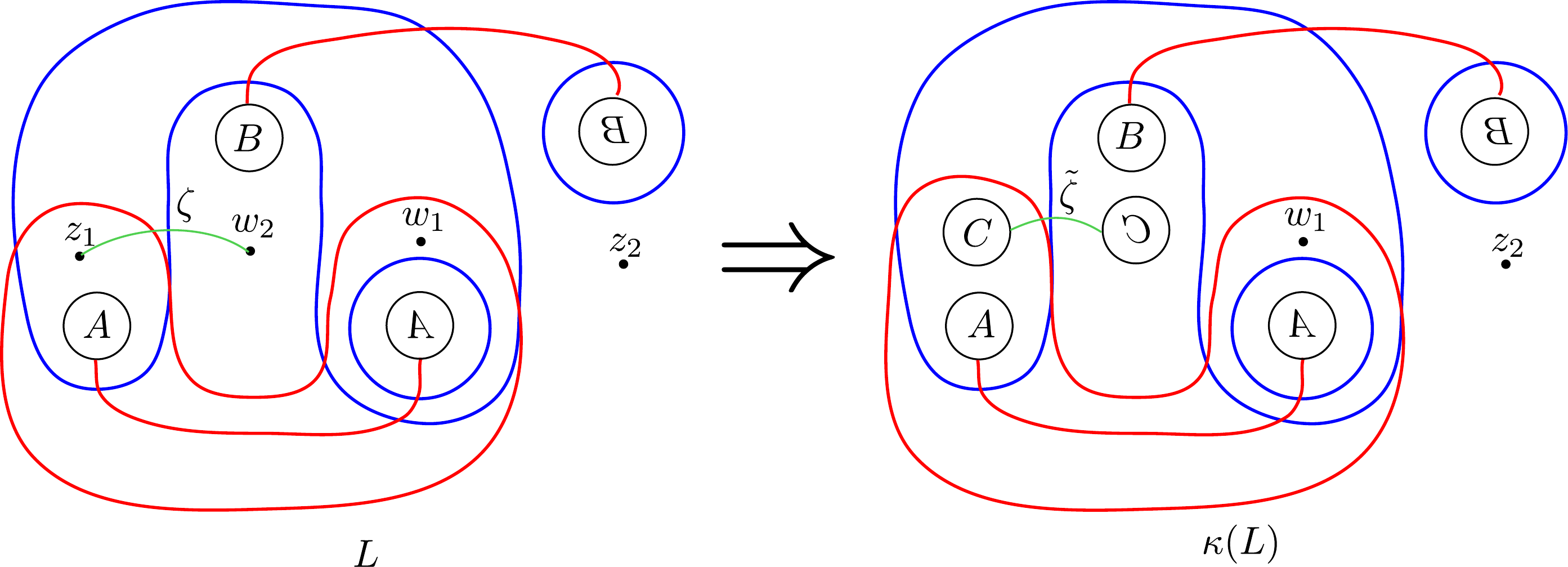}
    \caption{Illustration of the knotification process and closure construction  on Heegaard diagrams. The relative 1-cycle $\zeta$ connecting $z_1$ to $w_2$ is replaced with a closed loop $\tilde{\zeta}$ along the added 1-handle.}
    \label{fig:closure}
    \end{figure}
\end{center}
\begin{thm}\label{thm:knotification}
Under the identification of $\widehat{HFL}(Y,L)$ with $\widehat{HFL}(\kappa(Y,p,q), \kappa(L,p,q))$,
the homological action of the relative 1-chain $\zeta$ on $\widehat{HFL}(Y,L)$ is identified with the homological action of the closure $\tilde{\zeta}$ on $\widehat{HFL}(\kappa(Y,p,q), \kappa(L,p,q))$.
\begin{proof}
Let $(\Sigma, \bm{\alpha},\bm{\beta},\mathbf{w},\mathbf{z})$ be a Heegaard diagram for $(Y,L)$ and $L_1$(resp. $L_2$) be the link component where $p$(resp. $q$) lies on. Suppose that the arcs connecting $w_i$ with $z_i$ form $L_i$. Then by attaching a 1-handle connecting $p$ to $q$ on $\Sigma$ and disregarding $z_1$ and $w_2$, we obtain a Heegaard diagram $(\Sigma', \bm{\alpha}, \bm{\beta}, \mathbf{w}-\left\{w_2\right\}, \mathbf{z}-\left\{z_1\right\})$ for $(\kappa(Y, \left\{p,q\right\}), \kappa(L,\left\{p,q\right\}))$, where $\Sigma'$ is $\Sigma$ with the 1-handle attached. Note that the two Heegaard diagrams share the same set of generators $\mathbb{T}_\alpha\cap\mathbb{T}_\beta$, and the identification of $\widehat{HFL}(Y,L)$ with $\widehat{HFL}(\kappa(Y,p,q), \kappa(L,p,q))$ is induced from the identification of the generators.

We can choose complex structures on $\Sigma$ and $\Sigma'$ so that they coincide on every domain except those containing $z_1$ or $w_2$. Then in the formula
\begin{align}\label{eq:1}
    \tilde{\zeta}\cdot \mathbf{x} =     \sum_{\mathbf{y}\in\mathbb{T}_\alpha\cap\mathbb{T}_\beta}\sum_{\phi\in\pi_2(\mathbf{x},\mathbf{y})}\# \Mo(\phi)e_{\tilde{\zeta}}(\phi)\mathbf{y},
\end{align}
the contribution of the holomorphic disks $\phi$ with multiplicity 0 on the 1-handle and the corresponding holomorphic disks on formula \eqref{eq:0} is the same, as the intersection number $e_\zeta(\phi)=\zeta\cdot \partial_\alpha \phi$ is the same as $e_{\tilde{\zeta}}(\phi)$ and the count of moduli space $\# \Mo(\phi)$ is the same as well.

The other type of holomorphic disks, that is, those passing through the 1-handle region, may appear on the formula \eqref{eq:1}, but by stretching the necks of the 1-handle \cite[Proof of Theorem 1.1]{ozsvathHolomorphicDisksLink2008} one can ignore this case. More precisely, they showed that we can find a complex structure on $\Sigma'$ with the property that for all $\mathbf{x},\mathbf{y}\in \mathbb{T}_\alpha\cap \mathbb{T}_\beta$ and for all $\phi\in\pi_2(\mathbf{x},\mathbf{y})$ appearing in the formula for the differential, if $\# \Mo(\phi)\neq 0$, we must have $n_{z_1}(\phi)=n_{w_2}(\phi)=0$.
\end{proof}
\end{thm}
Now we state the homological-action-enhanced version of the disjoint union formula for $\widehat{HFL}$, compare with Theorem \ref{thm:connsum}.
\begin{thm}\label{thm:split}
Let $L$ be a split link $L=L_1\sqcup L_2$. Then we have a K\"{u}nneth type formula for $\widehat{HFL}$ with homological action in the sense that
\begin{align*}
    \widehat{HFL}(L)\cong \widehat{HFL}(L_1)\otimes \widehat{HFL}(L_2)\otimes \mathbb{F}[X]/(X^2)
\end{align*}
as a $\mathbb{F}[H_1(S^3,L)]$-module, where the module structure on the right is the tensor product of $\mathbb{F}[H_1(S^3,L_1)]$-module $\widehat{HFL}(L_1)$, $\mathbb{F}[H_1(S^3,L_2)]$-module $\widehat{HFL}(L_2)$, and $\mathbb{F}[X]$-module $\mathbb{F}[X]/(X^2)$. Here the isomorphism
\begin{align}\label{eq:split}
    H_1(S^3,L)\cong H_1(S^3,L_1)\oplus H_1(S^3,L_2)\oplus \ZZ
\end{align}
is used implicitly and the tensor product of $\FF[G_i]$-module $M_i, i\in 1,2$ is defined as
\begin{align*}
    (X_1\oplus X_2)\cdot (m_1\otimes m_2)=(X_1\cdot m_1)\otimes m_2+m_1\otimes (X_2\cdot m_2)
\end{align*}
for $X_i\in G_i, m_i\in M_i, i=1,2$.
\begin{proof}
Let $(\Sigma_i,\bm{\alpha}_i,\bm{\beta}_i,\mathbf{w}_i, \mathbf{z}_i)$ be a Heegaard diagram for $L_i, i=1,2$. Then we can take the connected sum of the two Heegaard diagram where the two feet of connected sum neck are near a point $z\in\mathbf{z}_1$ and $w\in\mathbf{w}_2$, respectively, to form a Heegaard diagram for $L$. Note that the core of the connected sum neck (from $z$ to $w$) spans $H_1(S^3,L)/\left(H_1(S^3,L_1)\oplus H_1(S^3,L_2)\right)\cong \ZZ$.

From the decomposition \eqref{eq:split}, there are two mutually distinct cases to consider:
\begin{itemize}
    \item $\zeta$ is in one of $\widehat{HFL}(S^3,L_i)$, $i=1,2$:\\
    In these cases, we can find a 1-cycle contained in the connected summand $(\Sigma_i, \bm{\alpha}_i, \bm{\beta}_i, \mathbf{w}_i, \mathbf{z}_i)$. Hence it is clear that $\zeta$ acts trivially on the other two tensorands.
    \item $\zeta$ is a multiple of the core of the connected sum neck:\\
    In this case, the core of the connected sum neck has trivial image in each connected summand $(\Sigma_i, \bm{\alpha}_i, \bm{\beta}_i, \mathbf{w}_i, \mathbf{z}_i)$, hence acts trivially on $\widehat{HFL}(L_i)$. And it is easy to check that it acts as multiplication by $X$ on $\mathbb{F}[X]/(X^2)$,
\end{itemize}
\end{proof}
\end{thm}

\subsection{Twisted coefficients}
In this section, we review the construction of link Floer homology with twisted coefficients, a link Floer analog of the well-known construction in Heegaard Floer homology(cf. \cite[Chapter 8]{ozsvathHolomorphicDisksThreeManifold2004}). The content of this subsection should be familiar to experts in Heegaard Floer theory, but to the author's best knowledge generalization to link Floer homology (and to sutured Floer homology in particular) is new.

In this section, we consider a general sutured manifold and its Heegaard diagrams, and the result on link Floer homology follows as a special case of the corresponding result on sutured Floer homology.

For a sutured Heegaard diagram $(\Sigma, \bm{\alpha},\bm{\beta},\mathbf{w})$ representing the balanced sutured manifold $Y$ and $\mathbf{x}_0, \mathbf{x},\mathbf{y}\in \mathbb{T}_\alpha\cap \mathbb{T}_\beta$, the set of homotopy classes of Whitney disks $\pi_2(\mathbf{x},\mathbf{y})$ is isomorphic to $\pi_2(\mathbf{x}_0,\mathbf{x}_0)$. One such isomorphism is induced from a choice of a complete set of paths $\left\{\phi_i\right\}$ \cite{ozsvathHolomorphicDisksTopological2004}, a set of Whitney disks $\phi_i\in\pi_2(\mathbf{x}_0, \mathbf{x}_i)$ with the property that for any $i,j$,
\begin{align*}
    \phi_i * \pi_2(\mathbf{x}_i,\mathbf{x}_j)\cong \pi_2(\mathbf{x}_0,\mathbf{x}_0)*\phi_j
\end{align*}
holds. This induces a (surjective) additive assignment in the sense of \cite[Definition 2.12]{ozsvathHolomorphicDisksTopological2004}
\begin{align*}
    A:\pi_2(\mathbf{x},\mathbf{y})\to H_2(Y;\ZZ),
\end{align*}
ie., $A(\phi*\psi)=A(\phi)+A(\psi)$ holds for any composable pair of Whitney disks $(\phi, \psi)$.

More precisely, this is given by the composition of the isomorphism $\pi_2(\mathbf{x},\mathbf{y})\cong \pi_2(\mathbf{x}_0,\mathbf{x}_0)$ and the natural homomorphism $H:\pi_2(\mathbf{x}_0,\mathbf{x}_0)\to H_2(Y;\ZZ)$ from the periodic domains to $H_2(Y;\ZZ)$ as defined in e.g. \cite[Definition 3.9]{juhaszHolomorphicDiscsSutured2006}.

The additive assignment $A$ induced from a complete set of paths as above is universal in the sense that any $G$-valued additive assignment $A':\pi_2(\mathbf{x},\mathbf{y})\to G$ factors through $A$. That is, one can find a homomorphism $H_2(Y;\ZZ)\to G$ such that pre-composition with $A$ gives rise to $A'$. As the homomorphism $H_2(Y;\ZZ)\to G$ (or, equivalently, the $\mathbb{F}[H_2(Y;\ZZ)]$-module structure on $G$ defined by the homomorphism) contains all the information on the additive assignment $A'$, we often abuse the notations and call the homomorphism $H_2(Y;\ZZ)\to G$ or the $\mathbb{F}[H_2(Y;\ZZ)]$-module structure on $G$ as the additive assignment.

For any $G$-valued additive assignment $A'$, we define a chain complex $\underline{SFC}(Y;A')$ freely generated over $\mathbb{F}[H_2(Y;\ZZ)]$ by $\mathbb{T}_\alpha\cap \mathbb{T}_\beta$ and the differential given by
\begin{align*}
    \partial \mathbf{x}=\sum_{\mathbf{y}\in\mathbb{T}_\alpha\cap \mathbb{T}_\beta}\left(\sum_{\phi\in \pi_2(\mathbf{x},\mathbf{y})}\#\Mo(\phi) e^{A'(\phi)}\mathbf{y}\right),
\end{align*}
whose homology is the sutured Floer homology $\underline{SFH}(Y;A')$ of $Y$ with twisted coefficients $A'$. The sum in the differential is finite under the exactly same admissibility condition as required for the ordinary sutured Floer homology $SFH(Y)$ of $Y$, and the isomorphism class of $\underline{SFH}(Y;A')$ is an invariant of $Y$ and $A'$. (The proof for the invariance of $\underline{\widehat{HF}}$ in \cite{ozsvathHolomorphicDisksThreeManifold2004} applies verbatim to our case.)

When $A'$ is the universal additive assignment $A:\pi_2(\mathbf{x}_0, \mathbf{x}_0)\to H_2(Y;\ZZ)$, we omit the twisted coefficient $A'$ from the notation and call it the \textit{totally twisted} sutured Floer homology. The sutured Floer homology $\underline{SFH}(Y;A')$ with an arbitrary additive assignment $A'$ is related to the totally twisted homology $\underline{SFH}(Y)$ by the following formula:
\begin{align*}
    \underline{SFH}(Y;A')\cong H_*\left(\underline{SFC}(Y)\otimes_{\mathbb{F}[H_2(Y)]}A'\right).
\end{align*}

A special case of our interest is the following:
\begin{ex}\label{ex:twisted coefficients}
Let $\zeta$ be a relative 1-cycle on $\Sigma$. Then, as in the previous section, $\zeta$ defines a $\ZZ$-valued additive assignment $e_\zeta:\pi_2(\mathbf{x},\mathbf{y})\to \ZZ$ defined by
\begin{align*}
    e_\zeta(\phi)=\zeta\cdot \partial_\alpha \phi,   
\end{align*}
where $\partial_\alpha\phi=(\partial \phi)\cap \mathbb{T}_\alpha$ is the multi-arc in $\Sigma$ that lies on $\mathbf{\alpha}$. In terms of the equivalence between $\ZZ$-valued additive assignments and homomorphisms from $H_2(Y)$ to $\ZZ$, the corresponding homomorphism is given as 
\begin{align*}
    e_\zeta(a) = \zeta\cdot \partial_\alpha a,
\end{align*}
where $\partial_\alpha a$ is a 1-cycle on $\Sigma$ defined as
\begin{align*}
    \partial_\alpha a = \sum_{i}(a\cdot a_i) \alpha_i,
\end{align*}
where $a_i$ is the co-core of the 2-handle attached along $\alpha_i$ in the construction of $Y$ from the Heegaard diagram $(\Sigma, \bm{\alpha},\bm{\beta})$.

This defines a ring homomorphism $\mathbb{F}[H_2(Y)]\to \Lambda$, where $\Lambda$ is the universal Novikov ring
\begin{align*}
    \Lambda = \left\{\sum_{i=0}^\infty a_{r_i} t^{r_i}\mid r_i\in \mathbb{R}, a_{r_i}\in \mathbb{F}, r_i \to \infty\right\}.
\end{align*}
We interpret this homomorphism as a $\mathbb{F}[H_2(Y)]$-module structure over $\Lambda$ and denote it as $\Lambda_\zeta$. Tensoring with the totally twisted sutured Floer complex $\underline{SFC}(Y)$, we obtain a chain complex $\underline{SFC}(Y;\Lambda_\zeta)$ freely generated over $\Lambda_\zeta$.
\end{ex}
\begin{cor}
The isomorphism class of the homology group $\underline{SFH}(L;\Lambda_\zeta)$ depends only on the homology class $[\zeta]\in H_1(Y,\partial Y;\ZZ)/Tors\cong \mathrm{Hom}(H_2(Y),\ZZ)$ of $\zeta$.
\begin{proof}
This is an easy adaptation of \cite[Proof of Lemma 2.4]{niHomologicalActionsSutured2014}. Using the same notation as in \cite{niHomologicalActionsSutured2014}, the chain homotopy equivalence $f:\underline{SFC}(Y;\Lambda_{\omega_1})\to\underline{SFC}(Y;\Lambda_{\omega_2})$ is induced by
\begin{align*}
    f(\mathbf{x})=t^{\frac{n_{\mathbf{x}}(B')}{m}}\mathbf{x}.
\end{align*}
\end{proof}
\end{cor}

\subsection{Module structure over $A=\FF[X]/(X^2)$ and splitness detection theorem}
As homological actions square to zero, a relative 1-cycle $\zeta\in H_1(S^3, L)$ defines an $A$-module structure on $\widehat{HFL}(L)$, where $\zeta$ acts as multiplication by $X$.

In \cite{lipshitzKhovanovHomologyAlso2019}, a similar structure on the Khovanov homology of $L$ is considered. Their main theorem \cite[Theorem 1]{lipshitzKhovanovHomologyAlso2019} provides a connection between this truncated module structure and the splitness of the link. In this section, we prove the link Floer homology version of this splitness detection theorem, Theorem \ref{thm:splitdetection}.
\begin{proof}[Proof of $(1)\Rightarrow (2)$ in Theorem \ref{thm:splitdetection}]
Suppose that $L=L_1\sqcup L_2$ is split and $S$ is the embedded sphere in the link complement separating $L_1$ from $L_2$. This in particular means that a path $\zeta$ connecting $L_1$ to $L_2$ spans the $\mathbb{Z}$ summand in the decomposition \eqref{eq:split} in Theorem \ref{thm:split}. Hence the $A$-module structure on $\widehat{HFL}(L)$ induced from $\zeta$ is the restriction of the $\FF[H_1(S^3,L)]$-module structure on the subring $\FF[X]$ spanned by $\zeta\in H_1(S^3, L)$ after truncation by $(X^2)$. But Theorem \ref{thm:split} states that multiplication by $X$ acts trivially on the $\widehat{HFL}(L_i)$ tensorands, $i=1,2$. As $A$ is clearly free over itself, $\widehat{HFL}(L)$ is free over $A$ as well.
\end{proof}
That is, the ``easy" direction is essentially the enhanced version of the disjoint union formula for $\widehat{HFL}$. But to prove the ``hard" direction of Theorem \ref{thm:splitdetection}, we need to introduce some notations and (deep) symplectic topological results in Heegaard Floer theory, which follows from now on.

In Example \ref{ex:twisted coefficients}, we consider the twisted complex $\underline{SFC}(Y;\Lambda_\zeta)$. The very same construction but using a real 2-form $\omega\in \mathrm{Hom}(H_2(Y;\RR),\RR)\cong H^2(Y;\RR)$ in place of $e_\zeta$ gives rise to a twisted complex $\underline{SFC}(Y;\Lambda_\omega)$ (compare with the twisted knot Floer complexes in \cite{niNonseparatingSpheresTwisted2013}).

The following lemma is the sutured analogue of \cite[Corollary 2.3]{alishahiBorderedFloerHomology2019}.
\begin{lm}\label{lm:uppersemicontinuity}
Let $\omega\in \Omega^2(Y)$ be a generic closed 2-form on $Y$. Then for any closed 2-form $\omega'\in \Omega^2(Y)$ we have
\begin{align*}
\mathrm{dim}_{\Lambda_{\omega'}}\left(\underline{SFH}(Y;\Lambda_{\omega'})\right)\geq \mathrm{dim}_{\Lambda_\omega}\left(\underline{SFH}(Y;\Lambda_\omega)\right).
\end{align*}
(A 2-form $\omega\in \Omega^2(Y)$ is generic if the evaluation map $e_\omega:H_2(Y)\to \RR, a\mapsto \int_a\omega$ is injective.)
\begin{proof}
We have chosen our notations to match with the notations in \cite[Section 2.1]{alishahiBorderedFloerHomology2019}. The entire Section 2.1 but replacing $\underline{\widehat{CF}}(Y)$ and $\underline{\widehat{HF}}(Y)$ with $\underline{SFC}(Y)$ and $\underline{SFH}(Y)$ respectively shows the lemma.
\end{proof}
\end{lm}
Finally, we need a rather standard lemma:
\begin{lm}
For a nonsplit link $L\subseteq S^3$, the knot complement of its knotification $\kappa(L)\subseteq \kappa(S^3)=\connsum^{l-1}(S^1\times S^2)$ is irreducible.
\begin{proof}
First note that an annulus $A$ obtained by puncturing the belt sphere of a 1-handle in $\kappa(S^3)$ along the intersection with $\kappa(L)$ is incompressible. As the core of $A$ is a meridian of $\kappa(L)$, its homotopy class in $\pi_1(\kappa(S^3)\setminus \kappa(L))$ is nontorsion. Hence $A$ is $\pi_1$-injective, ie, incompressible.

As surgering out along an incompressible surface preserves the irreducibility of the manifold, it suffices to show that the knot complement $\kappa(S^3)\setminus \kappa(L)$ cut along the belt sphere of the 1-handles is irreducible. But this is just the link complement $S^3\setminus L$.

It is left to check that the complement of a nonsplit link $L$ is irreducible. As $L$ is nonsplit, there are no separating essential spheres in $S^3\setminus L$. And there are no nonseparating essential spheres in $S^3\setminus L$ by Alexander's lemma. Hence $S^3\setminus L$ is irreducible, completing the proof.
\end{proof}
\end{lm}
\begin{proof}[Proof of $(2)\Rightarrow (1)$ in Theorem \ref{thm:splitdetection}]
We will show the contrapositive $\neg (1)\Rightarrow \neg (2)$ and closely follow the arguments in \cite[Section 5.2]{lipshitzKhovanovHomologyAlso2019}. Assume that $L$ is nonsplit.

First of all, using Theorem \ref{thm:knotification}, it suffices to show that $\widehat{HFL}(\kappa(L))$ is not free. As the knot complement $\kappa(S^3)\setminus \kappa(L)$ is irreducible by the previous lemma, a deep result in \cite[Theorem 3.8]{niNonseparatingSpheresTwisted2013} guarantees the existence of an open subset of 2-forms $\omega$ such that $\underline{\widehat{HFL}}(\kappa(L);\Lambda_\omega)$ is nontrivial. As an open subset is dense, Lemma \ref{lm:uppersemicontinuity} shows that $\underline{\widehat{HFL}}(\kappa(L);\Lambda_\omega)\neq 0$ for any 2-form $\omega$.

Let $\FF[t^{-1},t]_\omega$ be $\FF[t^{-1},t]$ viewed as an $\FF[H_2(\kappa(S^3)\setminus \kappa(L))]$-module via $\omega$, and similarly for $\FF(t)_\omega$. By the universal coefficient theorem, we have that
\begin{align*}
    \underline{\widehat{HFL}}(\kappa(L);\FF(t)_\omega)&\cong \underline{\widehat{HFL}}(\kappa(L);\FF[t^{-1},t]_\omega)\otimes_{\FF[t^{-1},t]}\FF(t)\\
    \underline{\widehat{HFL}}(\kappa(L);\Lambda_\omega)&\cong \underline{\widehat{HFL}}(\kappa(L);\FF(t)_\omega)\otimes_{\FF(t)}\Lambda
\end{align*}
Hence $\underline{\widehat{HFL}}(\kappa(L);\Lambda_\omega)\neq 0$ implies that $\underline{\widehat{HFL}}(\kappa(L);\FF[t^{-1},t]_\omega)$ has positive rank over the ring $\FF[t^{-1},t]$.

But the relation between $\widehat{HFL}(\kappa(L))$ and $\underline{\widehat{HFL}}(\kappa(L);\FF[t^{-1},t]_\omega)$ \cite[Theorem 5]{lipshitzKhovanovHomologyAlso2019} then implies that the $A$-module structure of $\widehat{HFL}(\kappa(L))$ induced by $\omega$ decomposes as
\begin{align*}
    \widehat{HFL}(\kappa(L))\cong \FF^m\oplus \FF\left<z_1,\cdots, z_k, w_1,\cdots, w_k\right>,
\end{align*}
where $X$ acts trivially on $\FF^m$ and on $\FF\left<w_1,\cdots, w_k\right>$ and $m>0$.\\
(\cite[Theorem 5]{lipshitzKhovanovHomologyAlso2019} is stated in terms of $\widehat{HF}$ and $\underline{\widehat{HF}}$ but its argument applies verbatim to our case.)

In particular, the kernel of multiplication by $X$ has dimension $m+k$, while the dimension of $\widehat{HFL}(\kappa(L))$ is $m+2k$. This contradicts the freeness of $\widehat{HFL}(\kappa(L))$ as the dimension of a free $A$-module is twice that of its kernel.
\end{proof}

\section{Shape of the link Floer polytopes}
In this section, we consider the constraints on the link Floer polytope from the convex geometric point of view. See e.g. \cite{grunbaum2013convex} for the standard definitions and notations in convex geometry used throughout this section.
\begin{defi}
The \textit{link Floer polytope} $P(L)$ of a link $L$ is a convex hull of the set of Alexander multigradings of $\widehat{HFL}(L)$, as a subset of $\mathbb{H}(L)$. More precisely, the link Floer polytope of $L$ is
\begin{align*}
    P(L):=\mathrm{conv}\left(\left\{h\in \mathbb{H}(L)\middle| \widehat{HFL}(L,h)\neq 0\right\}\right).
\end{align*}
\end{defi}
The link Floer homology of $L$ and the topology of the complement of $L$ are closely related by the following theorem. 
\begin{thm}[{\cite[Theorem 1.1]{ozsvathLinkFloerHomology2008}}]\label{thm:minkowski}
For a link without trivial components, its link Floer polytope (scaled up by a factor of 2) is equal to the Minkowski sum of the dual Thurston norm polytope of the link complement and the hypercube $[-1,1]^n$.
\end{thm}
For the sake of brevity, we often identify a property of a link complement and the corresponding property of a link and especially call the dual Thurston norm of the link complement of a link as the dual Thurston norm of the link itself.

Any convex polytope can be represented in two ways:\\
The H-representation, a system of linear inequalities of the form
\begin{align*}
    \left\{\mathbf{x}\in \mathbb{R}^l\mid A\mathbf{x}\leq \mathbf{b}, A\in \RR^d\times \RR^l, \mathbf{b}\in \RR^d\right\},
\end{align*}
or the V-representation,
\begin{align*}
    \mathrm{conv}(V):=\left\{\sum_i \lambda_i v_i\mid 0\leq \lambda_i, \sum_i \lambda_i=1, v_i\in V\right\},
\end{align*}
where $V$ is a finite collection of points in $\RR^l$. The V-representation is just a convex hull description of the polytope. The H-representation is describing the polytope as a (finite) intersection of closed half-spaces $A_i\mathbf{x}\leq b_i$, where $A_i\in \RR^d, b_i\in \RR, 1\leq i \leq d$. A closed half-space containing the polytope and intersecting nontrivially with the polytope on the boundary hyperplane is called a \textit{supporting half-space} and the boundary hyperplane is called \textit{supporting hyperplane}. The two notions are nearly equivalent:\\
Choosing an orientation of a supporting hyperplane determines which side of the hyperplane to be taken as the corresponding supporting half-space.

It is often convenient to visualize the link Floer homology group $\widehat{HFL}(L,h)$ as the collection of dots on $h\in P(L)$ of as many as the rank $\rank (\widehat{HFL}(L,h))$ of the link Floer homology. Then the spectral sequence in Theorem \ref{thm:CFSS} can be visualized as follows:\\
First we reduces the dimension of the link Floer polytope by one, by passing through the canonical projection $\mathbb{H}(L)\to \mathbb{H}(L\setminus L_i)$ of the form 
\begin{align*}
    (x_1,\cdots,x_i,\cdots, x_l)\mapsto (x_1,\cdots, \widehat{x_i},\cdots, x_l),    
\end{align*}
 following the grading shift 
\begin{align*}
    (x_1,\cdots, x_l)\mapsto (x_1-\frac{1}{2}\lk(L_i,L_1),\cdots, x_l-\frac{1}{2}\lk(L_i,L_l)).
\end{align*}
Then the dots over $h\in \mathbb{H}(L\setminus L_i)$ are cancelled out according to the differential $D$ in the spectral sequence in Theorem \ref{thm:CFSS} to yield the link Floer homology of the sublink. See Figure \ref{fig:CFSS} for example.
\begin{center}
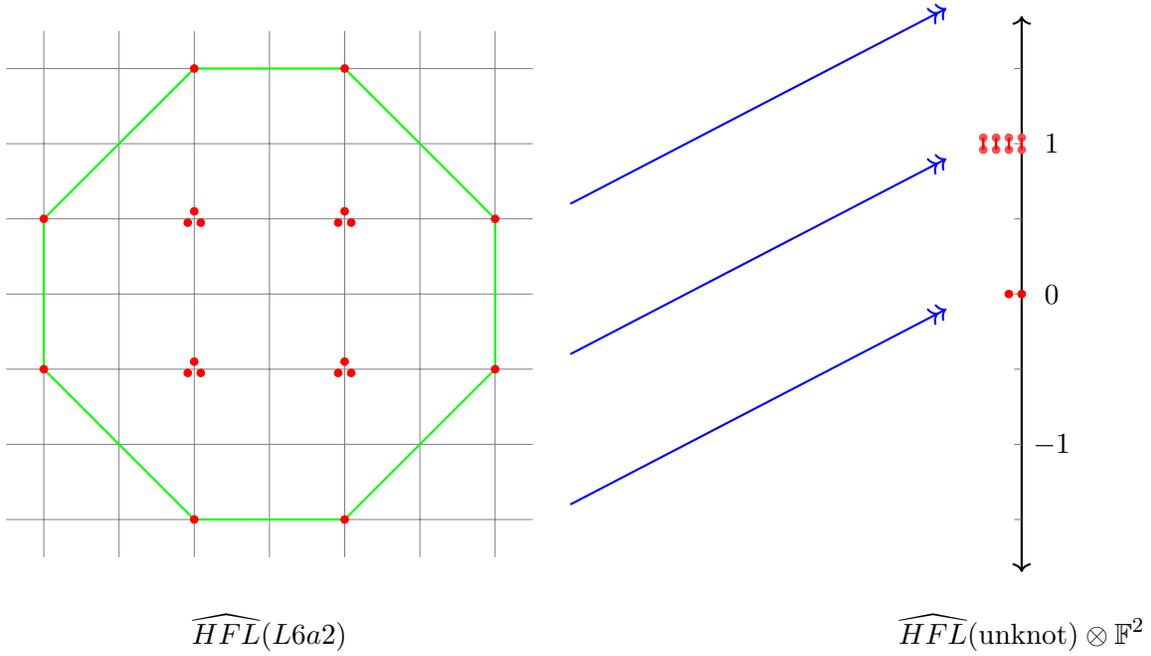
\begin{figure}[ht]
\vspace{1.5em}
\begin{tikzpicture}
\draw[step=1cm, gray, very thin] (-13.5, -3.5) grid (-6.5, 3.5);
\draw[green, thick] (-7, 1) -- (-9, 3) -- (-11, 3) -- (-13, 1) -- (-13, -1) -- (-11, -3) -- (-9, -3) -- (-7, -1) -- cycle;
\foreach \i in {-1,...,1}
\foreach \x in {-1, 1}
\foreach \y in {-1, 1}
{
\draw[red, fill] ({-10+\x+0.1*sin(120*\i)}, {\y+0.1*cos(120*\i)}) circle [radius=0.05];
}
\foreach \x in {-1, 1}
\foreach \y in {-3, 3}
{
\draw[red, fill] ({-10+\x}, \y) circle [radius=0.05];
\draw[red, fill] ({-10+\y}, \x) circle [radius=0.05];
}

\draw[thick, <->] (0, -3.7) -- (0, 3.7);
\foreach \x in {-1,...,1}
{
\node (\x) at (0.4, {2*\x}) {$\x$};
}
\foreach \x in {-3,...,3}
{
\draw[gray] (0, \x) -- (-0.1, \x);
}
\foreach \x in {0,...,2}
{
\draw[->>, blue, thick] (-6,{-3+2*\x+0.2}) -- (-1,{2*\x-0.2});
}
\foreach \i in {0,...,1}
{
\draw[red, fill] ({\i*-0.17}, 0) circle [radius=0.05];
}
\foreach \i in {0,...,3}
\foreach \j in {-1,1}
{
\draw[red!70, fill] ({-0.17*\i},{2+0.08*\j}) circle [radius=0.05];
\draw[red, thick] ({-0.17*\i},{2+0.08}) -- ({-0.17*\i},{2-0.08});
}
\node (L6a2) at (-10, -4.5) {$\widehat{HFL}(L6a2)$};
\node (unknot) at (0, -4.5) {$\widehat{HFL}(\mathrm{unknot})\otimes \FF^2$};
\end{tikzpicture}
\vspace{0.5em}
\caption{Visualization of component removal spectral sequence of the link $L6a2$. 8 red dots over the Alexander grading $1$ of $\widehat{HFL}(\mathrm{unknot})\otimes \FF^2$ are pairwise cancelled out as indicated in red edges, whereas the 2 red dots over the Alexander grading $0$ are not cancelled.}
\label{fig:CFSS}
\end{figure}
\end{center}
Now we state a simple but  powerful corollary of Theorem \ref{thm:minkowski}, which is repeatedly used in the next section.

\begin{lm}\label{lm:minkowski}
The Minkowski sum $P+ Q$ of two (possibly degenerate) convex polytopes $P$ and $Q$ has at least $\mathrm{max}\left(|\mathrm{vert}(P)|, |\mathrm{vert}(Q)|\right)$ vertices, where $\mathrm{vert}(P)$(resp. $\mathrm{vert}(Q)$) is the set of vertices in $P$(resp. $Q$).
\begin{proof}
It is a standard fact that the Minkowski sum of two convex polytopes is the convex hull of the componentwise sums of vertices of the two polytopes, hence $\mathrm{vert}(P+Q)$ is a subset of $\mathrm{vert}(P)+\mathrm{vert}(Q)$. This can be shown as follows:\\
Let $v$ be a vertex of $P+Q$ and $H$ be a supporting hyperplane of $P+Q$ at $v$. We can furthermore assume that $H$ intersects $P+Q$ only at $v$. (The set of hyperplanes intersecting $P+Q$ only at $v$ is dense in the set of all supporting hyperplanes of $P+Q$ at $v$.) Suppose that $\mathbf{a}^T\mathbf{x}+b=0$ is the defining equation for $H$, and $P+Q$ lies on the negative side of $H$ in the sense that $\mathbf{a}^T\mathbf{x}+ b\leq 0$ for $\mathbf{x}\in P+Q$. Note that $v$ is characterised as the only point in $P+Q$ which maximizes the height function $h(\mathbf{x})=\mathbf{a}^T\mathbf{x}+b$. By the convexity of $P$(resp. $Q$), we can find a hyperplane $H'$ parallel to $H$ and supporting $P$(resp. $Q$) on the negative side. Then for any vertex $v_1\in \mathrm{vert}(P)$(resp. $v_2\in \mathrm{vert}(Q)$) on $H'$, $v_1$(resp. $v_2$) maximizes the height function $h$ on $P$(resp. $Q$). Hence for any such pair $(v_1, v_2)$, $v_1+v_2$ is a point in $P+Q$ maximizing $h$, thus must be equal to $v$.

We prove the lemma with a little tweak on the above argument. Suppose that $v$ is a vertex of $P$ and $H$ is a supporting hyperplane of $P$ at $v$. Again we can assume without loss of generality that $H$ intersects $P$ only at $v$. Then $v$ is a unique point in $P$ which maximizes the height function $h$ with respect to the hyperplane $H$. Perturbing $H$ a bit if necessary, we can also assume that there is a unique point $w$ in $Q$ which maximizes $h$. Then $v+w$ is the only point in $P+Q$ which maximizes $h$, hence the map $\mathrm{vert}(P)\to \mathrm{vert}(P+Q), v\mapsto v+w$ (which may depend on the choice of hyperplanes $H$) is injective, implying that $|\mathrm{vert}(P)|\leq |\mathrm{vert}(P+Q)|$. Switching the ole of $P$ and $Q$, we also have the inequality $|\mathrm{vert}(Q)|\leq |\mathrm{vert}(P+Q)|$, proving the lemma.
\end{proof}
\end{lm}

\begin{cor}\label{cor:numvert}
For a $l$-component link $L$ without trivial components, $P(L)$ has at least $2^l$ vertices. Moreover, the minimum is attained only by a centrally symmetric box whose edges are parallel to the coordinate axes.
\begin{proof}
The first statement comes from Lemma \ref{lm:minkowski} and Theorem \ref{thm:minkowski}, thus only the second statement requires a proof.

Let $P$ be the dual Thurston norm polytope of $L$ such that the link Floer polytope $P+ [-1,1]^l$ has $2^l$ vertices. Recall the proof of Lemma \ref{lm:minkowski} and especially the construction of the injective map $\mathrm{vert}([-1,1]^l)\hookrightarrow\mathrm{vert}(P+[-1,1]^l)$ therein. What we have shown there can be reformulated as follows:\\
For a vertex $v$ of $[-1,1]^l$, let $C(v)\subseteq \mathrm{vert}(P)$ be the set of vertices $w\in\mathrm{vert}(P)$ for which there exists a hyperplane $H$ such that $v$(resp. $w$) uniquely maximizes the height function $h$ with respect to $H$ within $[-1, 1]^l$(resp. $P$). Then there exists an injection  $\mathrm{vert}([-1,1]^l)\hookrightarrow \mathrm{vert}(P+[-1,1]^l)$ such that each $v$ is mapped into $v+C(v)$.

We also note that $\left\{v+C(v)\right\}_{v\in \mathrm{vert}([-1,1]^l)}$ as defined above is a partition of $\mathrm{vert}(P+[-1,1]^l)$. This is because for each vertices $v+w$, the vertices $v, w, v+w$ are the unique maximum points of a height function within $[-1,1]^l, P, P+[-1,1]^l$, respectively, thus there is no overlap between the sets $v+C(v), v\in \mathrm{vert}([-1,1]^l)$. As the number of vertices in $P+[-1,1]^l$ is the same as the number of vertices in $[-1,1]^l$, we conclude that $C(v)$'s are singleton sets. Let $w=w(v)\in \mathrm{vert}(P)$ be the unique vertex in $C(v)$. Then we claim that the solid angle of $P+[-1,1]^l$ with respect to $v+w$, ie., the intersection of $\mathrm{cone}_{v+w}(P+[-1,1]^l)$ and a small sphere centered at $v+w$, is contained in the solid angle of the hypercube $[-1,1]^l$ with respect to $v$; otherwise one can find a supporting hyperplane of $[-1,1]^l$ at $v$ which does not support $P+[-1,1]^l$ at $v+w$, but then there exists a vertex $w'$ of $P$ maximizing the height function with respect to $H$ (which is necessarily different from $w$) and hence $C(v)$ contains two different elements $w$ and $w'$, a contradiction.

Now we observe that the union of the solid angles of a convex polytope with respect to every vertex must cover the whole sphere. As the $2^l$ solid angles of the hypercube $[-1,1]^l$ form a partition of the sphere (to be more precise, the solid angles intersect with measure 0), the solid angles of $P+[-1,1]^l$ cannot be strictly smaller than the solid angles of $[-1,1]^l$. Hence each vertex $v+w$ of $P+ [-1,1]^l$ has the solid angle the same as that of the hypercube at the corresponding vertex $v$. Hence the Minkowski sum polytope $P+ [-1,1]^l$ must be a box whose edges are parallel to the coordinate axes. The central symmetry condition follows from the symmetry of the link Floer homology.
\end{proof}
\end{cor}

Another trivial observation is:
\begin{cor}
For a link without trivial components, its link Floer polytope is nondegenerate.
\begin{proof}
Any Minkowski sum with the hypercube $[-1,1]^l$ contains a copy of $[-1,1]^l$, whose dimension is $l$.
\end{proof}
\end{cor}
\begin{rmk}
The dual Thurston norm polytope of a link $L$ may be degenerate, even under the assumption that there are no trivial components in $L$.
\end{rmk}

\section{Proof of Theorem \ref{thm:main}}
\begin{defi}
Let $R_l$ be the second smallest rank of the link Floer homology among the set of all $l$-component links in $S^3$.
\end{defi}
\begin{rmk}
Due to Remark \ref{rmk:even} and Remark \ref{rmk:hopf}, $R_l\leq 2^l$ and $R_l$ is even for $l>1$.
\end{rmk}
For the rest of this section, we consider a link $L$ of $l>1$ components that realizes the second smallest link Floer rank. 

By Remark \ref{rmk:hopf}, $L$ must be a link of unknots. We may assume for the following proof that $L$ is nonsplit, using the disjoint union formula along with the induction on the number of components of $L$.\\
More precisely, if $L=L_1\sqcup L_2$ were split link, then from the disjoint union formula in Theorem \ref{thm:connsum} either $L_1$ or $L_2$ must be an unlink, otherwise we achieve smaller link Floer rank by replacing one of the nontrivial $L_i$ with the unlink. And if a link had the second smallest link Floer rank, any sublink obtained by removing an unlink sublink should also have the second smallest link Floer rank. Hence we can reduce it to links with a smaller number of components.

\begin{rmk}\label{rmk:relativethm}
Adding a link component increase the link Floer rank by at least twice in view of the spectral sequence in Theorem \ref{thm:CFSS} and the equality holds if the added component is trivial by the disjoint union formula. 
\end{rmk}

In particular, the link Floer polytope of $L$ has at least $2^l$ vertices. Note also that the proof of Corollary \ref{cor:numvert} applies equally well to the projection of the link Floer polytope onto the coordinate hyperplanes, as the projection of the cube is a cube, just one dimension lower.

We first treat the 2-component link case, which is the initial step of the induction for the general case.
\begin{prop}
$R_2=4$, and only the Hopf link realizes this rank.
\begin{proof}
It is easy to check that Hopf link has link Floer rank 4. By the previous remark, we have that $R_2=4$. Hence it suffices to show that Hopf link is the only link with link Floer rank 4.

Let $L$ be a 2-component nonsplit link with link Floer rank 4. By the symmetry and nondegeneracy of the link Floer polytope, the Alexander multigrading of the four generators of $\widehat{HFL}(L)$ are of the form $(a,b),(-a,b),(a,-b),(-a,-b)$ for some $a,b>0$
.

Let $\lk$ be the linking number of $L$. (A different choice of orientation may reverse the sign of $\lk$, but it does not change the following argument.) As the grading shift by $\frac{1}{2}\lk$ must send $\pm a$ and $\pm b$ to 0, the only Alexander grading of the unknot, we have that $a=b=|\frac{1}{2}\lk|$.

As there is a unique generator with the maximum total Alexander grading $|\mathrm{lk}|$, the fiberedness detection theorem\cite{niKnotFloerHomology2007} guarantees that $L$ is fibered and the genus of the fiber is $|\lk|-1$. If the genus $|\lk|-1$ were positive, the knot Floer homology in its next-to-top Alexander grading $\widehat{HFK}(L, |\lk|-1)$ must have been nontrivial by \cite{baldwinNoteKnotFloer2018}, but the knot Floer homology is supported exactly on $\pm |\lk|$ and $0$, a contradiction. Hence the linking number $\lk$ is 1 and the fiber is an annulus. This implies that $L$ must be a Hopf link, as the powers of the positive Dehn twist are the only possible monodromies on the annulus. As $p$-th power of the Dehn twist gives $L(p,1)$, the only possibility is the Hopf link in $S^3$ when $p=\pm 1$.
\end{proof}
\end{prop}
Now follows the proof for the general case using induction.
\begin{prop}
For $l>2$, $R_l=2^l$ and only the disjoint union of the Hopf link and the unlink realizes this rank.
\begin{proof}
We proceed by induction on the number of components $l$.

First we show that $R_l=2^l$. Suppose by contradiction that $R_l<2^l$. Removing a component, the $(l-1)$-component sublinks of $L$ must have link Floer rank $<2^{l-1}$, hence they are the $(l-1)$-component unlink by the induction hypothesis $R_{l-1}=2^{l-1}$. In particular, $L$ is algebraically split and thus the grading shifts in the spectral sequence in Theorem \ref{thm:CFSS} are all zero.

Now consider the spectral sequence in Theorem \ref{thm:CFSS}. As there is no shift in Alexander multigrading, the spectral sequence simply projects $P(L)$ onto the hyperplane $x_i=0$. Let $P'(L)$ be the image of $P(L)$ under this projection. The differential $D$ cancels out the dots of $P'(L)$ of the same Alexander multigrading by pairs to yield $\widehat{HFL}(L\setminus L_1)\otimes \FF^2$, hence there exists an even number of dots over each Alexander multigradings except possibly at the origin. A priori the parity of the number of dots over the origin is not determined, but as the sum of all the number of dots must be even, the parity of the number of dots over the origin is even as well.

Noting that $P'(L)$ is the Minkowski sum of the projection of the dual Thurston norm polytope and a hypercube $[-1,1]^{l-1}$, the proof of Corollary \ref{cor:numvert} applies verbatim to $P'(L)$, hence $P'(L)$ has at least $2^{l-1}$ different Alexander multigradings. In particular, $P'(L)$ has at least $2\times 2^{l-1}=2^l$ dots. But then the link Floer rank of $L$ must be at least $2^l$, a contradiction. Hence $R_l=2^l$.

Next, we show that only the disjoint union of the Hopf link and the unlink realizes the link Floer rank $2^l$. Consider the link Floer polytope $P(L)$. By Corollary \ref{cor:numvert}, $P(L)$ is a centrally symmetric box whose edges are parallel to the coordinate axes. By the induction hypothesis, the $(l-1)$-component sublinks of $L$ are either unlink or disjoint union of the Hopf link and the unlink.

Suppose first that one of the $(l-1)$-component sublinks is the unlink. Consider the spectral sequence in Theorem \ref{thm:CFSS} to the link Floer homology of this unlink. As the projection of $P(L)$ along each coordinate axis is 2-1 and only the origin, the unique Alexander multigrading of the link Floer homology of the unlink, is left after the cancellation from the spectral sequence in Theorem \ref{thm:CFSS}, the differential $D$ from the spectral sequence in Theorem \ref{thm:CFSS} cancels out all the dots except the two which are sent to the origin. Hence the sublink has link Floer rank at most $2/2=1$, ie, the sublink must be an unknot. Hence $l=2$, which is already covered in the previous proposition.

Now suppose that all the sublinks are the disjoint union of the Hopf link and the unlink. As the link Floer rank of the disjoint union of the Hopf link and the $(l-3)$-component unlink is $2^{l-1}$, the spectral sequence in Theorem \ref{thm:CFSS} degenerates for any choice of link component $L_1\subseteq L$ and each of the projected polytopes $P'(L)$ is a symmetric box by the second statement of Corollary \ref{cor:numvert}. From the grading shift formula, this implies that the linking numbers of $L$ are all zero; otherwise the shift will break the central symmetry. But the linking number between the two components of the Hopf link is nonzero, a contradiction.
\end{proof}
\end{prop}
\section{Equality in rank inequalities}
As we observed in the previous section, for a fixed number of components, the problem of classifying all the links with small link Floer rank reduces to the classification of links of minimal link Floer rank within the category of nonsplit links. In this section, we provide some approaches on the latter problem.

One useful argument is to keep track of the change of link Floer rank under the spectral sequence of Theorem \ref{thm:split}. As the link Floer rank is a measurement of the complexity of the link, if the spectral sequence significantly decrease the link Floer rank, there is a good chance that we may replace the removed component with a simpler knot to obtain a link with smaller link Floer rank. On the other hand, if the spectral sequence does not decrease the link Floer rank, then the removed component is linked to the rest of the link in simplest way in a sense. Formally:

\begin{question}
When the spectral sequence in Theorem \ref{thm:CFSS} becomes trivial? Equivalently, when the equality holds in the following rank inequality?
\begin{align*}
    \widehat{HFL}(L)\geq 2\cdot \widehat{HFL}(L\setminus L_1)
\end{align*}
\end{question}
The disjoint union formula provides complete answer to the above question when $L_1$ is unlinked from the rest of $L$. Hence we are mainly interested in nonsplit links.

\begin{proof}[Proof of Proposition \ref{prop:relative}]
The first statement follows from the grading shift formula. The $i$-th Alexander grading is shifted by $-\frac{1}{2}\lk(L_1,L_i)$ under the spectral sequence, hence all the linking numbers $\lk(L_1, L_i)$ must be zero to ensure the central symmetry of the link Floer homology.

The second statement is a simple corollary of the splitness detection theorem, Theorem \ref{thm:splitdetection}. We first observe that homological action commutes with the differential $D$:\\
For a generator $\mathbf{x}\in \mathbb{T}_\alpha\cap\mathbb{T}_\beta$ and a relative 1-cycle $\zeta$, 
\begin{align*}
    D(\zeta\cdot\mathbf{x})=\zeta\cdot D\mathbf{x}.
\end{align*}
Hence the homological action induces an action on $\widehat{HFL}(L\setminus L_1)$. To identify this action with the homological action on the sublink $L\setminus L_1$, it is best to extend the definition of homological action to generalized Heegaard diagrams. (cf. \cite[Chapter 3]{manolescuHeegaardFloerHomology2010})

In a generalized Heegaard diagram, we allow $\mathbf{z}$ to have less points than $\mathbf{w}$. Such a Heegaard diagram is called \textit{link-minimal} in \cite{manolescuHeegaardFloerHomology2010}), and the Heegaard diagrams considered so far is called \textit{minimally-pointed}. This data uniquely determines a link in the same manner as before, the only difference is that there are some $\mathbf{w}$ basepoints (called \textit{free} basepoints) which are not paired with $\mathbf{z}$ basepoints to form a link component. All the definition and proof on the link Floer homology and homological action work equally well with respect to the generalized Heegaard diagrams, with the caveat that we do not allow relative 1-cycles to have ends on the free basepoints.

More precisely, a link-minimal Heegaard diagram $\mathcal{H}'$ is obtained by applying free zero/three stabilizations to a minimally-pointed Heegaard diagram $\mathcal{H}$. Each time we apply a free zero/three stabilization, the link Floer homology $\widehat{HFL}(\mathcal{H}')$ is doubled:
\begin{align*}
    \widehat{HFL}(\mathcal{H}')\cong\widehat{HFL}(\mathcal{H})\otimes \FF^2.
\end{align*}
Moreover, up to appropriate Heegaard moves, the free zero/three stabilization can be done in the region away from the relative 1-cycle $\zeta$. Hence the action of $\zeta$ on $\widehat{HFL}(\mathcal{H}')\cong \widehat{HFL}(\mathcal{H})\otimes \mathbb{F}^2$ is the tensor product of the action of $\zeta$ on $\widehat{HFL}(\mathcal{H})$ and the trivial action on $\FF^2$.

In this framework, the spectral sequence in Theorem \ref{thm:CFSS} is the result of removing a $z$ basepoint from a Heegaard diagram $\mathcal{H}=(\Sigma, \bm{\alpha}, \bm{\beta}, \mathbf{w},\mathbf{z})$ of $L$, where the remaining $w$ basepoint becomes a free basepoint. Hence it is clear that the induced action of $\zeta$ on $\widehat{HFL}(L\setminus L_1)\otimes \FF^2$ coincides with the action of $\zeta$ on $\widehat{HFL}(\mathcal{H})$.

Hence the spectral sequence is equivariant under the action of $\zeta$. But the splitness detection theorem says that if $L\setminus L_1$ is split, the truncated module $\widehat{HFL}(L\setminus L_1)$ is free with respect to the path connecting two split compoents (and so is $\widehat{HFL}(L\setminus L_1)\otimes \FF^2$). As $\widehat{HFL}(L)$ is not free with respect to any relative 1-cycle by the nonsplit hypothesis, this is a contradiction.
\end{proof}

\bibliographystyle{amsalpha}
\bibliography{bib}

\end{document}